\documentclass[10pt,reqno]{amsart}
\usepackage{a4wide}
\usepackage{amssymb}
\allowdisplaybreaks[1]
\DeclareMathOperator{\C}{C}

\DeclareMathOperator{\cop}{cop}
\DeclareMathOperator{\G}{G}
\DeclareMathOperator{\M}{M} \DeclareMathOperator{\Mor}{Mor}
\DeclareMathOperator{\B}{B} 
 \DeclareMathOperator{\Aut}{Aut}
\DeclareMathOperator{\W}{L}
\DeclareMathOperator{\X}{X}
\DeclareMathOperator{\Ad}{Ad}
\DeclareMathOperator{\Ind}{Ind}

\DeclareMathOperator*{\tens}{\otimes}
\newtheorem{twr}{Theorem}[section]
\newtheorem{lem}[twr]{Lemma}
\newtheorem{stwr}[twr]{Proposition}
\theoremstyle{definition}
\newtheorem{defin}[twr]{Definition}

\newtheorem{uw}[twr]{Remark}

\newcommand{\id}{{\rm id}}

\newcommand{\mbG}{\mathbb{G}}
\newcommand{\mbX}{\mathbb{X}}
\newcommand{\mbD}{\mathbb{D}}
\newcommand{\mbB}{\mathbb{B}}
\begin{document}
\title[]{Rieffel Deformation of Homogeneous Spaces}
\author{P.~Kasprzak}
\address{Department of Mathematical Sciences, University of Copenhagen}
\address{On leave from Department of Mathematical Methods in Physics, Faculty of Physics, Warsaw University}
\thanks{Supported by the Marie Curie Research Training Network Non-Commutative Geometry MRTN-CT-2006-031962}
\thanks{Supported by Geometry and Symmetry of Quantum Spaces,
PIRSES-GA-2008-230836}
\email{pawel.kasprzak@fuw.edu.pl}
\begin{abstract}
Let $\G_1\subset \G$ be a closed subgroup of a locally compact group $\G$ and let $\X=\G/\G_1$ be  the quotient space of left cosets. Let $\mbX=(\C_0(\X),\Delta_{\X})$ be the corresponding $\mbG$-$\C^*$-algebra where $\mbG=(\C_0(\G),\Delta)$. Suppose that $\Gamma$ is a closed  abelian subgroup of $\G_1$ and let $\Psi$ be a $2$-cocycle on the dual group $\Hat\Gamma$. Let $\mbG^\Psi$ be the Rieffel deformation of $\mbG$. Using the results of paper \cite{Kasp2} we may construct $\mbG^\Psi$-$\C^*$-algebra  $\mbX^\Psi$ - the Rieffel deformation of $\mbX$. On the other hand we may perform the Rieffel deformation of the subgroup  $\G_1$ obtaining the closed quantum subgroup $\mbG^\Psi_1\subset\mbG^\Psi$, which in turn, by the results of \cite{Vaes}, leads to the  $\mbG^\Psi$-$\C^*$-algebra $\mbG^\Psi/\mbG_1^\Psi$. In this paper we show that  $\mbG^\Psi/\mbG_1^\Psi\cong\mbX^\Psi$. We also consider the case where $\Gamma\subset \G$ is not a subgroup of $\G_1$, for which we cannot construct the subgroup $\mbG^\Psi_1$. Then $\mbX^\Psi$ cannot be identified with a quantum quotient. What may be shown is that  it is a $\mbG^\Psi$-simple object in the category of $\mbG^\Psi$-$\C^*$-algebras. \end{abstract}
\maketitle \tableofcontents
\begin{section}{Introduction}
The theory of locally compact quantum groups (LCQG) has already reached its maturity. Almost ten years have passed since the appearance of the seminal paper of J. Kustermans and S. Vaes \cite{KV}, where  the axiomatic theory was formulated. A locally compact quantum group is a $\C^*$-bialgebra  $(A,\Delta)$ equipped with a left and a right Haar weight $\phi$ and $\psi$. Imposing some natural conditions on $\Delta$ (weak cancellation) and on the Haar weights $\phi$ and $\psi$ (KMS-type conditions) the authors were able to develop the theory, proving among other things the existence of the coinverse  which admits the polar decomposition  $\kappa=R\circ\tau_{i/2}$ and showing that the theory is self dual. The assumption of the existence of the Haar weights, which is a theorem for the classical groups and for the compact quantum groups, may be perceived as a drawback.  It seems that a Haar weights free axiomatization is out of reach. There exists second formulation of the LCQG theory, due to T. Masuda, Y. Nakagami and S.L. Woronowicz \cite{MNW}, in which the authors include the existence of the coinverse $\kappa$ in the axioms. But to develop the theory they still need to assume the existence of a left Haar weight. In fact, it can be shown that both theories are equivalent.

One way of constructing examples of  LCQGs is to start with a classical group $\G$ and search for its deformations $\mbG_q$. In general,  the link between $\G$ and $\mbG_q$ is not rigorously described. There is at least one mathematical procedure - the Rieffel deformation - where this correspondence is clear. For the original approach of Rieffel we refer to \cite{Rf1}. In this paper we shall use our recent approach to the Rieffel deformation which describes it in terms of crossed product construction  (see \cite{Kasp}). The Rieffel deformation of $\G$ will be denoted by $\mbG^\Psi$ and its dual will be denoted by $\widehat{\mbG}^\Psi$. The deformation procedure of a locally compact group in terms of the transition from $\widehat\mbG$ to $\widehat{\mbG}^\Psi$ is described in Section \ref{twgr}. 

Let $\G$ be a locally compact group and $D$ a $\G$ - $\C^*$-algebra. Using the results of \cite{Kasp2} one may apply the Rieffel deformation to $D$, obtaining
 the deformed $\mbG^\Psi$ - $\C^*$-algebra $D^\Psi$. The concise account of the deformation procedure of $\G$ - $\C^*$-algebras is the subject of Section \ref{rco}. In Section \ref{prelimin} beside giving some preliminaries on $\mbG$-$\C^*$-algebras we also discuss the  notion of a quantum homogeneous space and the $\C^*$-algebraic quotient space $\mbG/\mbG_1$, a construction due to Vaes \cite{Vaes}. His construction may be performed for any closed quantum subgroup $\mbG_1$ of a regular LCQG $\mbG$ - for regularity we refer to \cite{BS}. The relation of $\mbG/\mbG_1$ with the induction procedure of the regular corepresentation is explained. The reader's awareness of this relation is crucial in the understanding of the  proof that $\mbX^\Psi\cong\mbG^\Psi/\mbG_1^\Psi$. In Section \ref{indregsec} we perform the induction procedure of the regular corepresentation $W_1^\Psi$ of the twisted group $\mbG_1^\Psi$ and compare the resulting objects with their untwisted counterparts. 
This enables us to prove that $\mbX^\Psi\cong\mbG^\Psi/\mbG_1^\Psi$ which is the subject of Section \ref{rdhsec}. Finally, in  Section \ref{nonq} we comment on the case where $\mbX^\Psi$ is not of the quotient type. In connection with it we show that  the Rieffel deformation of a $\mbG$-simple $\C^*$-algebra is $\mbG^\Psi$-simple.

The particular case of the Rieffel deformation of a homogeneous space  has been discussed in \cite{Varr}. In this paper J. Varilly treats the situation where there is given a pair $\Gamma\subset \G_1\subset \G$ of closed subgroups with $\Gamma$ being abelian and $\G$ compact. He  shows that it is possible to perform a covariant deformation of $\X=\G/\G_1$ obtaining a quantum homogeneous $\mbG^\Psi$-space $\mbX^\Psi$. In this specific situation the difficulties that one encounters in general do not manifest themselves.  

Throughout the paper we will freely use the language of
$\C^*$-algebras and the theory of locally compact quantum groups.
For the notion of multipliers and morphism of
$\C^*$-algebras we refer the reader to \cite{W4}.
For the theory of locally compact quantum groups we refer to
\cite{KV} and \cite{MNW}. 

I would like to express my gratitude to W. Szyma\'nski for many stimulating discussions, which greatly influenced the final form of this paper. 
\end{section}
\begin{section}{Preliminaries}\label{prelimin}
\begin{subsection}{$\mbG$-$\C^*$-category}\label{gcat}
Let us fix a notation related with  a locally compact quantum group $\mbG$. The $\C^*$-algebra and the comultiplication of $\mbG$ will be denoted by $\C_0(\mbG)$ and $\Delta_{\mbG}$ respectively. The von Neumann algebra associated with $\mbG$ and the Hilbert space obtained by the GNS-representation related with the left Haar weight  will be denoted respectively by $\W^\infty(\mbG)$ and $\W^2(\mbG)$. By $\widehat{\mbG}$ we shall denote the reduced version of the dual of $\mbG$. The modular conjugations related with the left Haar weight on $\mbG$ (on $\widehat\mbG$) will be denoted by $J$ (by $\Hat J$).

The main subject of this paper is related with the quantum group actions on quantum spaces. The following definition may be traced back to \cite{Pod}. To formulate it we adopt the following notation: a closed linear span of a subspace $\mathcal{W}\subset\mathcal{V}$ of a Banach space $\mathcal{V}$ will be denoted by $[\mathcal{W}]$. 
\begin{defin}\label{concoact} Let $\mbG$ be a LCQG. A $\mbG$-$\C^*$-algebra is a  pair $\mbD=(D,\Delta_D)$ consisting of a $\C^*$-algebra $D$ and a coaction $\Delta_D\in\Mor(D,\C_0(\mbG)\otimes D)$:
\[(\iota\otimes\Delta_D) \Delta_D = (\Delta_{\mbG}\otimes\iota)\Delta_D \] such that $[\Delta_D(D)(\C_0(\mbG)\otimes 1)]=\C_0(\mbG)\otimes D$. The $\C^*$-algebra $D$ will also be denoted by $\C_0(\mbD)$ and the coaction $\Delta_D$ will be denoted by $\Delta_{\mbD}$.
\end{defin} 
\begin{uw}\label{clcoac} Suppose that $\mbG$ corresponds to an ordinary locally compact group $\G$. There is a 1-1 correspondence between $\mbG$-$\C^*$-algebras and $\G$-$\C^*$-algebras, i.e. $\C^*$-algebras equipped with a continuous action of $\G$. In order to describe it we introduce the characters $\chi_g:\C_0(\G)\rightarrow\mathbb{C}$ that are associated with the points of $\G$: $\chi_g(f)=f(g)$, for any $g\in\G$ and $f\in\C_0(\G)$. For ang $\mbG$-$\C^*$-algebra $\mbD$ we define the corresponding continuous action $\alpha:\G\rightarrow\Aut(\C_0(\mbD))$ by the formula: $\alpha_g(a)=(\chi_{g^{-1}}\otimes\iota)\Delta_{\mbD}(a)$ for any $g\in\G$ and  $a\in\C_0(\mbD)$. 
\end{uw}
The category of $\G$-$\C^*$-algebras was consider either explicitly or implicitly in many different contexts - see for instance \cite{KQ} where it appears explicitly in the context of the Landstad duality.
In order to specify the category of $\mbG$-$\C^*$-algebras we adopt the following notion of $\mbG$-morphisms.
\begin{defin}\label{Gmor}
Let $\mbG$ be a LCQG and suppose that $\mathbb{B}$ and $\mbD$ are $\mbG$-$\C^*$-algebras. We say that a morphism $\pi:\C_0(\mbB)\rightarrow\C_0(\mbD)$ is a $\mbG$-morphism from $\mbB$ to $\mbD$ if $\Delta_{\mbD}\circ\pi=(\iota\otimes\pi)\circ\Delta_{\mbB}$. The set of $\mbG$-morphism from $\mbB$ to $\mbD$  will be denoted by $\Mor_{\mbG}(\mathbb{B},\mbD)$.
\end{defin}
\end{subsection}
\begin{subsection}{Quantum homogeneous spaces}\label{qhs} Let $\mbG$ be a locally compact quantum group and $\mbX$ be a $\mbG$ - $\C^*$-algebra.
The concept of a quantum homogeneous $\mbG$-space exists so far  only in the case of $\mbG$ being compact  and $\C_0(\mbX)$ being unital - see Definition 1.8 of \cite{Pod1}. In this case $\mbX$ is a quantum homogeneous space  if  $\Delta_{\mbX}$ is ergodic. Let $\mbG$ be a compact quantum group corresponding to a compact group $\G$. It may be checked that there is a 1-1 correspondence between  quantum homogeneous $\mbG$-spaces  with the underlying commutative $\C^*$-algebra $\C_0(\mbX)$ and the classical homogeneous $\G$-spaces. Unfortunately the ergodicity assumption in the non-compact case cannot serve as a replacement for homogeneity. The next approximation to the homogeneity is the notion of $\mbG$-simplicity:
\begin{defin}\label{homspdef} Let $\mbD$ be $\mbG$-$\C^*$-algebra. We say that $\mbD$ is $\mbG$-simple if for any $\mbG$-$\C^*$-algebra  $\mathbb{B}$ and any $\mbG$-morphism $\pi\in\Mor_{\mbG}(\mbD,\mbB)$ we have $\ker\pi=\{0\}$.  
\end{defin} 
In order to motivate the introduction the $\mbG$-simplicity let us prove that in the  case of the compact quantum groups the homogeneity of $\mbD$ implies its $\mbG$-simplicity.
\begin{lem}
Let $\mbG=(A,\Delta)$ be a compact quantum group with a faithful Haar measure $h$. Let $\mbX$ be a quantum homogeneous space with $\Delta_{\mbX}$ being injective. Then $\mbX$ is $\mbG$-simple.
\end{lem}
\begin{proof} Using the ergodicity of $\Delta_{\mbX}$ we may introduce the state on $\rho:\C_0(\mbX)\rightarrow\mathbb{C}$ such that \begin{equation}\label{rodef}(h\otimes\iota)\Delta_{\mbX}(a)=\rho(a) \mbox{ for any } a\in \C_0(\mbX).\end{equation} The faithfulness of $h$ and the injectivity of $\Delta_{\mbX}$ imply the  faithfulness of $\rho$. 
 Let $\pi\in\Mor_{\mbG}(\mbX,\mathbb{B})$ and suppose that $a\in\ker\pi$. Note that $(h\otimes\iota)\Delta_{\mbX}(a^*a)\in\ker\pi$: \[\pi((h\otimes\iota)\Delta_{\mbX}(a^*a))=(h\otimes\iota)\Delta_{\mbX}(\pi(a^*a))=0.\] On the other hand, using \eqref{rodef} we may see that $\pi((h\otimes\iota)\Delta_{\mbX}(a^*a))=\rho(a^*a)1$ which together with the above computation shows that $\rho(a^*a)=0$. The faithfulness of $\rho$ implies that $a=0$ hence $\ker\pi=\{0\}$.
\end{proof}
Let $\G$ be a locally compact quantum group and let $\mbG$ be the corresponding LCQG. 
It may be checked that there is a 1-1 correspondence between $\mbG$-simple $\C^*$-algebras with the  underlying $\C^*$-algebra being commutative  and minimal $\G$ - spaces - the $\G$-spaces with all orbit being dense (see p.49, \cite{DS}).  

As was already mentioned, a definition of a quantum homogeneous space appropriate for locally compact quantum groups is not yet known. 
What has been generalized so far is the notion of the quantum quotient space, due to Vaes \cite{Vaes}.  
But since the work of Podle\'s \cite{Pod} on quantum spheres we know that a generic quantum homogeneous space is not of the quotient type. The Rieffel of the homogeneous spaces deformation provides a new class of examples of non-compact  quantum spaces which in our opinion should also be considered as homogeneous and which generically are not of the quotient type. 

Let us move on to the discussion of the quantum quotient spaces. 
Let $\mbG$ be a regular LCQG (for the notion of regularity we refer to \cite{BS}). Let $\mbG_1$ be a closed quantum subgroup in the sense of Definition 2.5 \cite{Vaes}. As a part of the structure we have  the injective normal $*$-homomorphism  $\hat\pi:\W^\infty(\widehat{\mbG}_1)\rightarrow\W^\infty(\widehat{\mbG})$. Its commutant counterpart is denoted by $\hat\pi':\W^\infty(\widehat{\mbG}_1)'\rightarrow\W^\infty(\widehat{\mbG})'$. The definition of the measurable  quotient space $\W^\infty(\mbG/\mbG_1)$ goes as follows:
\[\W^\infty(\mbG/\mbG_1)=\{a\in \W^\infty(\mbG):a\hat\pi'(x)=\hat\pi'(x)a,\mbox{ for any }x\in\W^\infty(\widehat{\mbG}_1)'\}.\] A remarkable Theorem 6.1 \cite{Vaes}) provides  the existence and the uniqueness of the $\C^*$-algebraic version $\C_0(\mbG/\mbG_1)$ of $\W^\infty(\mbG/\mbG_1)$. For the needs of this paper we shall formulate it as a definition:
\begin{defin}\label{cqs}Let $\mbG$ be a LCQG which is regular and let $\mbG_1$ be its closed quantum subgroup. The $\C^*$-algebraic quotient $\mbG/\mbG_1$ is the unique $\mbG$-$\C^*$-algebra defined by the following conditions:
\begin{itemize}
\item $\C_0(\mbG/\mbG_1)\subset\W^\infty(\mbG/\mbG_1)$ is a strongly dense $\C^*$-subalgebra;
\item   $\Delta_{\mbG/\mbG_1}$ is given by the restriction of  $\Delta_{\mbG}$ to $\C_0(\mbG/\mbG_1)$;
\item  $\Delta_{\mbG}(\W^\infty(\mbG/\mbG_1))\subset\M(\mathcal{K}(\W^2(\mbG))\otimes \C_0(\mbG/\mbG_1))$ and the $*$-homomorphism \[\Delta_{\mbG/\mbG_1}:\W^\infty(\mbG/\mbG_1)\rightarrow\mathcal{L}(\W^2(\mbG)\otimes\C_0(\mbG/\mbG_1))\mbox{ is strict.}\]  
\end{itemize}
\end{defin}
For the notion of strictness we refer to Definition 3.1 of \cite{Vaes}. 
Note that the symbol  $\Delta_{\mbG/\mbG_1}$ denotes the coaction on $\C_0(\mbG/\mbG_1)$ as well as the strict map defined on $\W^\infty(\mbG/\mbG_1)$.
\begin{uw}\label{expr}
The difficulty of the proof of Theorem 6.1, \cite{Vaes} lies in the existence part. To explain the idea of the proof (which will be also important to understand this paper) let us consider the $\C^*$-algebra $\C_0(\widehat{\mbG}_1)$ treated as a Hilbert $\C^*$-module over itself. Performing the induction procedure of the regular corepresentation $W_1\in\M(\C_0(\mbG_1)\otimes\C_0(\widehat{\mbG}_1))$ we get the induced Hilbert module $\Ind(\C_0(\widehat{\mbG}_1))$ equipped with the additional structure: the induced corepresentation $\Ind(W_1)$ and the coaction $\gamma:\Ind(\C_0(\widehat{\mbG}_1))\rightarrow\M(\Ind(\C_0(\widehat{\mbG}_1))\otimes\C_0(\widehat{\mbG}))$ of the dual quantum group $\widehat{\mbG}$. The coaction  $\gamma$ is obtained by the induction procedure applied to the right coaction $\beta:\C_0(\widehat{\mbG}_1)\rightarrow\M(\C_0(\widehat{\mbG}_1)\otimes\C_0(\widehat{\mbG}))$,  $\beta(a)=(\iota\otimes\hat\pi)\Delta_{\widehat{\mbG}_1}(a)$ for any $a\in\C_0(\widehat{\mbG}_1)$.  The additional structure consisting of $\gamma$ and $\Ind(W_1)$ enables one to prove that the $\C^*$-algebra of compact operators $\mathcal{K}(\Ind(\C_0(\widehat{\mbG}_1))$ on $\Ind(\C_0(\widehat{\mbG}_1))$ is canonically isomorphic with the $\C^*$-algebra of a crossed product. The coaction $\gamma$ is identified with the dual coaction on the crossed product while $\Ind(W_1)$ is identified with the corepresentation implementing the coaction on the $\gamma$-invariants. The $\C^*$-algebra $\C_0(\mbG/\mbG_1)$ is defined as the Landstad-Vaes $\C^*$-algebra of  $\gamma$-invariants,  by which  we mean the $\C^*$-algebra satisfying the conditions of Theorem 6.7, \cite{Vaes}. 
\end{uw}
\end{subsection} 
\end{section}
\begin{section}{Group algebra twist}\label{twgr}
Let $\G$ be a locally compact group and let $\delta:\G\rightarrow\mathbb{R}_+$ be  the modular function. Suppose that $\Gamma$ is a closed abelian subgroup of $\G_1$, which in turn is a closed subgroup of $\G$. For reasons which will become clear later we shall assume that the modular function  $\delta:\G\rightarrow\mathbb{R}_+$ when restricted to $\Gamma$ is identically $1$: $\delta(\gamma)=1$ for any $\gamma\in\Gamma$. 

In order to perform the twisting procedure one has to fix a $2$-cocycle $\Psi$ on $\Hat\Gamma$, i.e. a  continuous function
  $\Psi:\hat{\Gamma}\times
\hat{\Gamma}\rightarrow
 \mathbb{T}^1$ satisfying:
\begin{itemize}
\item[(i)] $\Psi(e,\hat\gamma)=\Psi(\hat\gamma,e)=1$ for all $\hat\gamma\in\Hat\Gamma$;
\item[(ii)]$\Psi(\hat{\gamma}_1,\hat{\gamma}_2+\hat{\gamma}_3)
\Psi(\hat{\gamma}_2,\hat{\gamma}_3)
=\Psi(\hat{\gamma}_1+\hat{\gamma}_2,\hat{\gamma}_3)\Psi(
\hat{\gamma}_1,\hat{\gamma}_2) $ for all
$\hat{\gamma}_1,\hat{\gamma}_2,\hat{\gamma}_3\in\Hat\Gamma$.
\end{itemize}
For the theory of $2$-cocycles we refer to \cite{kl}.

In order to simplify the notation and make further computations less cumbersome we shall assume that $\Psi$ is a skew symmetric bicharacter on $\Hat\Gamma$:
\begin{align}
\label{bicheq}\Psi(\hat\gamma_1+\hat\gamma_2,\hat\gamma_3)=&\Psi(\hat\gamma_1,\hat\gamma_3)\Psi(\hat\gamma_2,\hat\gamma_3),\\
\Psi(\hat\gamma_1,\hat\gamma_2)=&\overline{\Psi(\hat\gamma_2,\hat\gamma_1)}.
\end{align}
We will only prove our results for such $2$-cocycles although we have reasons to believe that they hold for arbitrary $2$-cocycles.

The role of the bicharacter $\Psi\in\M(\C_0(\Hat{\Gamma})\otimes\C_0(\Hat{\Gamma}))$ will vary in the course of this paper. The simplest variation is connected with the identification $\C_0(\Hat{\Gamma})\cong\C^*(\Gamma)$, which enables us to treat $\Psi$ as an element of $\M(\C^*(\Gamma)\otimes\C^*(\Gamma))$. Furthermore, the morphism $\iota\in\Mor(\C^*(\Gamma),\C^*_l(\G_1))$ which corresponds to the representation  $\Gamma\ni\gamma\mapsto L_\gamma\in\M(\C_l^*(\G_1))$ enables us to treat $\Psi$ as an element of $\M(\C^*_l(\G_1)\otimes\C^*_l(\G_1))$. Finally, in some cases $\Psi$  will be treated as an operator acting on $\W^2(\mbG_1)\otimes\W^2(\mbG_1)$ or $\W^2(\mbG)\otimes\W^2(\mbG)$.  

Let us describe the twisting procedure of $(\C_l^*(\G),\Hat\Delta)$. In what follows  we shall use the notation of Section \ref{gcat}, denoting  the locally compact quantum group related to the pair $(\C_l^*(\G),\Hat\Delta)$ by $\widehat\mbG$. In particular, $\C_0(\widehat\mbG)=\C_l^*(\G)$ and $\Delta_{\widehat\mbG}=\Hat\Delta$. We may twist the comultiplication on $\Delta_{\widehat\mbG}$ by means of $\Psi\in\M(\C_0(\widehat\mbG)\otimes\C_0(\widehat\mbG))$:
$\Delta_{\widehat\mbG^\Psi}(a)=\Psi^*\Delta_{\widehat\mbG}(a)\Psi$. Using Corollary 5.3 of \cite{V} we see that there exists a locally compact quantum group $\widehat\mbG^\Psi$ such that $\C_0(\widehat\mbG^\Psi)=\C_0(\widehat\mbG)$, for which the comultiplication is given by $\Delta_{\widehat\mbG^\Psi}$. (For more general results concerning the twist of a quantum group by a $2$-cocycle we refer to \cite{DC}.) Furthermore, with our modular assumption $\widehat\mbG^\Psi$ is of the Kac type - the coinverse $\kappa_{\widehat{\mbG}^\Psi}$ is an involutive anti-automorphism.  The dual locally compact quantum group of $\widehat\mbG^\Psi$ will be denoted by $\mbG^\Psi$. Its description in terms of the Rieffel deformation was given in \cite{Kasp}. It is also of the Kac type - $\kappa_{\mbG^\Psi}^2=\id$ - which by Example 3.4 \cite{BS} implies that $\mbG^\Psi$ is regular. The only reason for the assumption $\delta|_\Gamma=1$ is to ensure the regularity of $\mbG^\Psi$ which is important in the construction of the quotient of a locally compact quantum group by its closed quantum subgroup.

In what follows we shall give the formula for the multiplicative operators $W^\Psi\in\W^\infty(\mbG^\Psi)\otimes\W^\infty(\widehat\mbG^\Psi)$ and $\Hat{V}^\Psi\in\W^\infty(\mbG^\Psi)'\otimes\W^\infty(\widehat{\mbG}^\Psi)$ where we adopted the notation of Section 2.1 of \cite{Vaes}. In order to do it let us introduce an element $F\in\M(\C_0(\Hat\Gamma)\otimes\C_0(\Hat\Gamma))$: $F(\hat\gamma_1,\hat\gamma_2)=\Psi(\hat\gamma_2,\hat\gamma_1+\hat\gamma_2)$. Using Theorem 1 of \cite{FV} and the the properties of  $\Psi$ we may see that the operator $W^\Psi$ acts on $\B(\W^2(\mbG)\otimes \W^2(\mbG))$ and it is of the form $W^\Psi=\Psi W\Psi'$ where \begin{equation}\label{psipr}\Psi'=(\hat{J}\otimes J)F(\hat{J}\otimes J).\end{equation} The assumption of the co-stability introduced in Section 2.7 of \cite{FV} is satisfied trivially in the discussed case. The required group homomorphism $t\mapsto \hat\gamma_t$ is the trivial homomorphism: $\hat\gamma_t=e\in\Hat\Gamma$ (we replaced the $\gamma_t$ introduced in \cite{FV} by $\hat\gamma_t$  which is more appropriate in our context). 

To give the formula for the multiplicative unitary $\Hat{V}^\Psi$ we only have to invoke the fact that it is the fundamental multiplicative unitary corresponding to the co-opposite quantum group $\widehat\mbG^\Psi_{\cop}$. The comultiplication of $\widehat\mbG^\Psi_{\cop}$ is the flip of the comultiplication $\Delta_{\widehat\mbG^\Psi}$. By the skew-symmetry of $\Psi$ we get
\begin{equation}\label{hatV}
\Hat{V}^\Psi=\Psi^* \Hat{V}\Psi'^*.
\end{equation} 

The twist construction that we applied to  $\Hat\mbG$ can be also applied to $\Hat\mbG_1$ leading to $\widehat{\mbG}_1^\Psi$ and the dual $\mbG_1^\Psi$. Note that we do not impose the modular condition on the embedding $\Gamma\subset\G_1$. Again using Theorem 1 of \cite{FV} one may see that $W_1^\Psi$ and $\Hat{V}_1^\Psi$ are related  to $W_1$ and $V_1$ in the way analogous to the case of $\mbG$ described above (see Eq. \eqref{hatV}). 
The aforementioned assumption of co-stability is also satisfied but the homomorphism $t\mapsto\hat\gamma_t$ may be non-trivial. In order to see that, we shall use the modular function $\delta_1:\G_1\rightarrow\mathbb{R}_+$. For any $t\in\mathbb{R}$ we may define the character $\hat\gamma_t\in\Hat\Gamma$ given by $\langle\hat\gamma_t,\gamma\rangle=\delta_1^{it}(\gamma)$. The homomorphism $t\mapsto\hat\gamma_t$ is the one that ensures the co-stability.
It may be checked that the quantum group $\mbG_1^\Psi$ is a closed quantum subgroup of $\mbG^\Psi$ in the sense of Definition 2.5 of \cite{Vaes}. In particular, there exists the embedding $\hat\pi:\W^\infty(\widehat\mbG_1^\Psi)\rightarrow\W^\infty(\widehat\mbG^\Psi)$. Its counterpart acting between commutants will be denoted by $\hat\pi':\W^\infty(\widehat\mbG_1^\Psi)'\rightarrow\W^\infty(\widehat\mbG^\Psi)'$. 
\end{section}
\begin{section}{Rieffel deformation of group coactions}\label{rco}
Suppose that $\G$ is a locally compact group containing an abelian closed subgroup $\Gamma$ and let $\Psi$ be a skew-symmetric bicharacter on $\Hat{\Gamma}$ (see Section \ref{twgr}).
Let $\mbX=(\C_0(\mbX),\Delta_{\mbX})$ be a $\mbG$-$\C^*$-algebra (see Definition \ref{concoact}).  In paper \cite{Kasp} we defined a $\mbG^\Psi$-$\C^*$-algebra $\mbX^\Psi$ - the Rieffel deformation of $\mbX$. The deformation procedure may be viewed as a two steps procedure: first extend $\Delta_{\mbX}$ to the morphism of the appropriate crossed products and then twist the extension by a unitary obtained from $\Psi$. Let us be more precise.

Let $\beta:\G\rightarrow\Aut(\C_0(\mbX))$ be the continuous action corresponding to the coaction $\Delta_\mbX$ (see Remark \ref{clcoac}).
Let $\alpha:\Gamma\rightarrow\Aut(\C_0(\mbX))$ be the restriction of $\beta$ to the subgroup $\Gamma$. The $\C^*$-algebra $\C_0(\mbX^\Psi)$ is defined as the Landstad algebra of the $\Gamma$-product $(\Gamma\ltimes_\alpha\C_0(\mbX),\lambda,\hat{\rho}^{\Psi})$ where $\hat{\rho}^{\Psi}:\Hat{\Gamma}\rightarrow\Aut(\Gamma\ltimes_\alpha\C_0(\mbX))$ is the $\Psi$-deformed dual action (see Section 3, \cite{Kasp}). Let us note that  $\Delta_\mbX\circ\alpha_\gamma=(\rho_{\gamma,e}\otimes\iota)\circ\Delta_\mbX$, where $\rho:\Gamma^2\rightarrow\Aut(\C_0(\mbG))$ is the action defined by 
$\rho_{\gamma_1,\gamma_2}(f)(g)=f(\gamma_1^{-1}g\gamma_2)$ for any $\gamma_1,\gamma_2\in\Gamma$ and $g\in \G$. The universal property of the crossed product  $\Gamma\ltimes_\alpha\C_0(\mbX)$ enables us to define the extension  $\Delta_\mbX^\Gamma\in\Mor(\Gamma\ltimes_\alpha\C_0(\mbX),\Gamma^2\ltimes_\rho\C_0(\mbG)\otimes\Gamma\ltimes_\alpha\C_0(\mbX))$ of $\Delta_\mbX$ to the level of crossed product. 

In order to describe the aforementioned twisting step we define  $\Upsilon\in\M(\Gamma^2\ltimes_{\rho}\C_0(\mbG)\otimes\Gamma\ltimes_\alpha\C_0(\mbX))$ as follows. Let $\Phi\in\Mor(\C^*(\Gamma^2),\Gamma^2\ltimes_{\rho}\C_0(\mbG)\otimes\Gamma\ltimes_\alpha\C_0(\mbX))$ be the morphism that corresponds to the representation $\Gamma\ni(\gamma_1,\gamma_2)\mapsto\lambda_{e,\gamma_1}\otimes\lambda_{\gamma_2}\in\M(\Gamma^2\ltimes_{\rho}\C_0(\mbG)\otimes\Gamma\ltimes_\alpha\C_0(\mbX))$. Applying it to $\bar\Psi$ we get $\Upsilon=\Phi(\bar\Psi)$. We may define $\Delta_{\mbX^\Psi}$ by the formula  $\Delta_{\mbX^\Psi}(b)=\Upsilon\Delta_{\mbX}^\Gamma(b)\Upsilon^*$ for any $b\in\Gamma\ltimes_\alpha\C_0(\mbX)$. By Theorems 4.3 and 4.4 of \cite{Kasp2}  $\Delta_{\mbX^\Psi}$ restricts to a morphism $\Delta_{\mbX^\Psi}:\C_0(\mbX^\Psi)\rightarrow\C_0(\mbG^\Psi)\otimes\C_0(\mbX^\Psi)$ which is a continuous coaction of $\mbG^\Psi$ on $\C_0(\mbX^\Psi)$. This defines $\mbG^\Psi$-$\C^*$-algebra $\mbX^\Psi$.

The reader may have noticed some differences between the above description of the deformation procedure of $\mbD$ and the one given in \cite{Kasp2}. They are due to the adopted definition of $\mbG^\Psi$ as the dual of $\widehat\mbG^\Psi$ and the fact  $\mbD$ is a left $\mbG$-$\C^*$-algebra whereas in \cite{Kasp2} we consider the right case.
\end{section}
\begin{section}{Induction of the regular corepresentation}\label{indregsec}
In this section we shall apply the induction procedure to the regular corepresentation $W_1^\Psi\in\M(\C_0(\mbG^\Psi_1)\otimes\C_0(\widehat\mbG^\Psi_1))$ of $\mbG_1^\Psi$ on $\C_0(\widehat\mbG^\Psi_1)$, where $\C_0(\widehat\mbG^\Psi_1)$ is treated as a $\C^*$-Hilbert module over itself. For the induction procedure in the framework of LCQGs we refer  to \cite{Vaes} (we shall also adopt the notation of this paper). As a result, we obtain the induced $\C_0(\widehat\mbG^\Psi_1)$-Hilbert module $\Ind(\C_0(\widehat\mbG_1^\Psi))$ together with the induced corepresentation $\Ind(W_1^\Psi)\in\mathcal{L}(\C_0(\mbG^\Psi)\otimes\Ind(\C_0(\widehat\mbG^\Psi_1)))$. We shall also apply the induction procedure to the right coaction $\beta^\Psi:\C_0(\widehat\mbG^\Psi_1)\rightarrow\M(\C_0(\widehat\mbG^\Psi_1)\otimes\C_0(\widehat\mbG^\Psi))$ of $\widehat\mbG^\Psi$ on $\C_0(\widehat\mbG^\Psi_1)$, defined by
\begin{equation}\label{bpsi}\beta^\Psi(a)=(\iota\otimes\hat{\pi})\Delta_{\widehat\mbG^\Psi_1}(a),\end{equation} where $\hat\pi\in\Mor(\C_0(\widehat\mbG^\Psi_1),\C_0(\widehat\mbG^\Psi))$ is the standard embedding. The result is the coaction  \[\Ind(\beta^\Psi):\Ind(\C_0(\widehat\mbG^\Psi_1))\rightarrow\M(\Ind(\C_0(\widehat\mbG^\Psi_1))\otimes\C_0(\widehat\mbG^\Psi)).\]
Finally, the induced objects $\Ind(\C_0(\widehat\mbG_1^\Psi)),\Ind(W_1^\Psi)$ and $\Ind(\beta^\Psi)$ will be compared with their untwisted counterparts $\Ind(\C_0(\widehat\mbG_1)),\Ind(W_1)$, and $\Ind(\beta)$. 

Let us first recall that the von Neumann algebra  of the twisted quantum group $\widehat{\mbG}^\Psi$ remains unchanged $\W^\infty(\widehat\mbG^\Psi)=\W^\infty(\widehat\mbG)$. This implies that the imprimitivity bimodule $\mathcal{I}^\Psi$, which is defined by
\[\mathcal{I}^\Psi=\{v\in\B(\W^2(\mbG_1),\W^2(\mbG))|\,\,vm=\hat\pi'(m)v,\mbox{ for any }m\in \W^\infty(\widehat\mbG^\Psi_1)'\}\] stays undeformed: $\mathcal{I}^\Psi=\mathcal{I}$. Nevertheless, the coaction $\alpha_{\mathcal{I}^\Psi}:\mathcal{I}^\Psi\rightarrow\mathcal{I}^\Psi\otimes\W^\infty(\widehat\mbG^\Psi)$:\,\, $\alpha_{\mathcal{I}^\Psi}(v)=\Hat{V}^\Psi(v\otimes 1)(\iota\otimes\hat{\pi})\Hat{V}_1^{\Psi*}$ gets twisted.  The relation with its untwisted counterpart is established by the following formula:
\begin{equation}\label{coactI}\alpha_{\mathcal{I}^\Psi}(v)=\Psi^*\alpha_{\mathcal{I}}(v)\Psi.\end{equation} Let us analyze the strict $*$-homomorphism $\pi_{\widehat\mbG^\Psi_1}:\W^\infty(\widehat\mbG^\Psi_1)\rightarrow\mathcal{L}(\W^2(\mbG))\otimes\C_0(\widehat\mbG^\Psi_1))$  (in Lemma 4.5, \cite{Vaes} it was denoted by $\pi_l$), given by:
\[\pi_{\widehat\mbG^\Psi_1}(m)=(\hat{\pi}\otimes\iota)(W_1^\Psi(m\otimes 1)W_1^{\Psi*})\] for any $m\in\W^\infty(\widehat\mbG^\Psi_1)$. It may be noted that the map $\pi_{\widehat\mbG^\Psi_1}$ coincides with $(\hat\pi\otimes\id)\Delta_{\widehat\mbG^\Psi_{1,\cop}}$ when appropriately interpreted.  Its relation with  $\pi_{\widehat\mbG_1}:\W^\infty(\widehat\mbG_1)\rightarrow\mathcal{L}(\W^2(\mbG))\otimes\C_0(\widehat\mbG^\Psi_1))$ is expressed by the twisting formula:
\[\pi_{\widehat\mbG^\Psi_1}(m)=\Psi\pi_{\widehat\mbG_1}(m)\Psi^*\] for any $m\in\W^\infty(\widehat\mbG^\Psi_1)=\W^\infty(\widehat\mbG_1)$. 
Using $\pi_{\widehat\mbG^\Psi_1}$ we may introduce the $\C_0(\widehat\mbG^\Psi_1)$ - module \[\mathcal{F}^\Psi=\mathcal{I}^\Psi\tens_{\pi_{\widehat\mbG^\Psi_1}}\W^2(\mbG)\otimes \C_0(\widehat\mbG^\Psi_1).\] It is equipped with the strict $*$-homomorphism $\pi_l^\Psi:\W^\infty(\widehat\mbG^\Psi)\rightarrow\mathcal{L}(\mathcal{F}^\Psi)$ and the strict $*$ - anti-homomorphism $\pi_r^\Psi:\W^\infty(\widehat\mbG^\Psi)\rightarrow\mathcal{L}(\mathcal{F}^\Psi)$, which are defined as follows:
\begin{align*}\pi_l^\Psi(m)(i\tens_{\pi_{\widehat\mbG^\Psi_1}}h\otimes a_1)&=(mi)\tens_{\pi_{\widehat\mbG^\Psi_1}}h\otimes a_1\\\pi_r^\Psi(m)(i\tens_{\pi_{\widehat\mbG^\Psi_1}}h\otimes a_1)&=i\tens_{\pi_{\widehat\mbG^\Psi_1}}\hat{J}m^*\hat{J}h\otimes a_1
\end{align*} for any $m\in\W^\infty(\widehat\mbG^\Psi)$, $i\in\mathcal{I}^\Psi$, $h\in \W^2(\mbG)$ and $a_1\in\C_0(\widehat\mbG^\Psi_1)$. In order to compare $\mathcal{F}^\Psi$ and $\mathcal{F}$ we prove:
\begin{stwr}\label{prop5.1} There exists a unitary transformation $U\in\mathcal{L}(\mathcal{F},\mathcal{F}^\Psi)$ such that
\[U(i\tens_{\pi_{\widehat\mbG_1}}h\otimes a_1)=i\tens_{\pi_{\widehat\mbG^\Psi_1}}\Psi(h\otimes a_1)\] for any $i\in\mathcal{I}$, $h\in \W^2(\mbG)$ and $a_1\in\C_0(\widehat\mbG^\Psi_1)$.  $U$ intertwines $\pi_l$ with $\pi_l^\Psi$ and $\pi_r$ with $\pi_r^\Psi$.
\end{stwr}
\begin{proof} For the existence of $U$ it is enough to note that:
\begin{align*}U(ia\tens_{\pi_{\widehat\mbG_1}}h\otimes a_1)&=ia\tens_{\pi_{\widehat\mbG^\Psi_1}}\Psi(h\otimes a_1)\\
&=i\tens_{\pi_{\widehat\mbG^\Psi_1}}\Psi\pi_{\widehat\mbG_1}(a)(h\otimes a_1)\\
&=U(i\tens_{\pi_{\widehat\mbG_1}}\pi_{\widehat\mbG_1}(a)(h\otimes a_1)).
\end{align*}
The fact that $U$ is unitary and that it possesses the required intertwining properties can be verified by a straightforward computation.
\end{proof}
Let us now introduce the coaction $\alpha_{\mathcal{F}^\Psi}:\mathcal{F}^\Psi\rightarrow\M(\mathcal{F}^\Psi\otimes \C_0(\widehat\mbG^\Psi))$. It is defined as a product of the coaction $\alpha_{\mathcal{I}^\Psi}$ given by \eqref{coactI} and  the coaction $\alpha_{\W^2(\mbG)\otimes\C_0(\widehat\mbG^\Psi_1)}:\W^2(\mbG)\otimes\C_0(\widehat\mbG^\Psi_1)\rightarrow \M(\W^2(\mbG)\otimes\C_0(\widehat\mbG^\Psi_1)\otimes\C_0(\widehat\mbG^\Psi))$ given by: \[\alpha_{\W^2(\mbG)\otimes\C_0(\widehat\mbG^\Psi_1)}(\xi)=\hat{V}^\Psi_{13}(\xi\otimes 1),\] for any $\xi\in\W^2(\mbG)\otimes\C_0(\widehat\mbG^\Psi_1)$. The formula defining $\alpha_{\mathcal{F}^\Psi}$ goes as follows:
\[\alpha_{\mathcal{F}^\Psi}(i\tens_{\pi_{\widehat\mbG^\Psi_1}}\xi)=\alpha_{\mathcal{I}^\Psi}(i)\tens_{\pi_{\widehat\mbG^\Psi_1}\otimes\iota}\alpha_{\W^2(\mbG)\otimes\C_0(\widehat\mbG^\Psi_1)}(\xi).\]
The coaction  $\alpha_{\mathcal{F}^\Psi}$ corresponds to the corepresentation  $Y^\Psi\in\mathcal{L}(\mathcal{F}^\Psi\otimes\C_0(\widehat\mbG^\Psi))$ of $\widehat\mbG^\Psi$ on $\mathcal{F}^\Psi$:
\[Y^\Psi(f\otimes a)=\alpha_{\mathcal{F}^\Psi}(f)(1\otimes a).\] In order to compare $Y^\Psi$ with its undeformed counterpart $Y\in\mathcal{L}(\mathcal{F}\otimes\C_0(\widehat\mbG))$ we compute:
\begin{align*}\allowdisplaybreaks
(U^*\otimes 1)Y^\Psi(U\otimes 1)(i\tens_{\pi_{\widehat\mbG_1}}\xi\otimes a)&\\&\hspace{-1cm}=(U^*\otimes 1)Y^\Psi(i\tens_{\pi_{\widehat\mbG^\Psi_1}}\Psi\xi\otimes a)\\&\hspace{-1cm}=(U^*\otimes 1)(\alpha_{\mathcal{F}^\Psi}(i\tens_{\pi_{\widehat\mbG^\Psi_1}}\Psi\xi))(1\otimes a)\\&\hspace{-1cm}=(U^*\otimes 1)(\Psi^*\alpha_{\mathcal{I}}(i)\Psi\tens_{\pi_{\widehat\mbG^\Psi_1}\otimes\iota}\Psi^*_{13}\Hat{V}_{13}\Psi'^*_{13}\Psi_{12}(\xi\otimes 1))(1\otimes a)\\&\hspace{-1cm}=(\Psi^*\alpha_{\mathcal{I}}(i)\Psi\tens_{\pi_{\widehat\mbG_1}\otimes\iota}\Psi_{12}^*\Psi^*_{13}\Hat{V}_{13}\Psi'^*_{13}\Psi_{12}(\xi\otimes 1))(1\otimes a)\\&\hspace{-1cm}=(\Psi^*\alpha_{\mathcal{I}}(i)\Psi\tens_{\pi_{\widehat\mbG_1}\otimes\iota}\Psi_{12}^*\Psi^*_{13}\Psi_{12}\Psi_{23}^*\Hat{V}_{13}\Psi'^*_{13}(\xi\otimes 1))(1\otimes a)\\&\hspace{-1cm}=(\Psi^*\alpha_{\mathcal{I}}(i)\tens_{\pi_{\widehat\mbG_1}\otimes\iota}\Psi_{13}\Psi_{23}\Psi_{12}^*\Psi^*_{13}\Psi_{12}\Psi_{23}^*\Hat{V}_{13}\Psi'^*_{13}(\xi\otimes 1))(1\otimes a)
\\&\hspace{-1cm}=(\Psi^*\alpha_{\mathcal{I}}(i)\tens_{\pi_{\widehat\mbG_1}\otimes\iota}\Hat{V}_{13}\Psi'^*_{13}(\xi\otimes 1))(1\otimes a)
\\&\hspace{-1cm}=(\pi_l\otimes\id)(\Psi^*)Y(\pi_r\otimes\id)(\Psi'^*)(i\tens_{\pi_{\widehat\mbG_1}}\xi\otimes a)
\end{align*}
In the fifth equality we used the easy to verify formula $\Hat{V}_{13}\Psi_{12}\Hat{V}^*_{13}=\Psi_{12}\Psi_{23}^*$ and in the sixth equality we used: 
$(\pi_{\widehat\mbG_1}\otimes\iota)\Psi=\Psi_{13}\Psi_{23}$. For the formula for $\Psi'$ we refer to \eqref{psipr}.
This computation shows that the relation between the corepresentations $Y\in\mathcal{L}(\mathcal{F}\otimes\C_0(\widehat\mbG))$ and $Y^\Psi\in\mathcal{L}(\mathcal{F}^\Psi\otimes\C_0(\widehat\mbG^\Psi))$ is given by \[(U^*\otimes\id)Y^\Psi(U\otimes\id)=(\pi_{l}\otimes\id)(\Psi^*)Y(\pi_r\otimes\id)(\Psi'^*),\] 
where we treat the anti-homomorphism $\pi_r:\W^\infty(\widehat\mbG)\rightarrow \mathcal{L}(\mathcal{F})$ as the homomorphism of the commutant $\W^\infty(\widehat\mbG)'$ (note that $\Psi'\in\W^\infty(\widehat\mbG)'\otimes\W^\infty(\widehat\mbG)$). 

We are now ready to compare the induced module $\Ind(\C_0(\widehat\mbG^\Psi_1))$ and the induced corepresentation $\Ind(W^\Psi_1)$ with their untwisted counterparts $\Ind(\C_0(\widehat\mbG_1))$ and $\Ind(W_1)$. Let us first recall the construction of the untwisted objects. By the results of \cite{Vaes} there exists a (unique) strict $*$-homomorphism $\Theta:\B(\W^2(\mbG))\rightarrow\mathcal{L}(\mathcal{F})$ such that 
\begin{align*}(\Theta\otimes\iota)\hat{V}&=Y,\\
\Theta(\hat{J}x^*\hat{J})&=\pi_r(x),
\end{align*} for any $x\in\W^\infty(\widehat\mbG)$. 
The induced $\C^*$-module is defined as
\[\Ind(\C_0(\widehat\mbG_1))=\{v:\W^2(\mbG)\rightarrow\mathcal{F}:vx=\Theta(x)v\mbox{ for all }x\in\B(\W^2(\mbG))\},\] where by assumption $v:\W^2(\mbG)\rightarrow\mathcal{F}$ is a continuous map.
It follows from the above definition that we may define the unitary transformation $\Phi:\mathcal{F}\rightarrow \W^2(\mbG)\otimes\Ind(\C_0(\widehat\mbG_1))$ of $\C_0(\widehat\mbG_1)$-modules, such that
\begin{equation}\label{phid}\Phi^*(h\otimes v)=v(h)\end{equation} for any $h\otimes v\in \W^2(\mbG)\otimes\Ind(\C_0(\widehat\mbG_1))$.
The induced corepresentation $\Ind(W_1)\in\mathcal{L}(\C_0(\mbG)\otimes\Ind(\C_0(\widehat\mbG_1)))$ of $\mbG$  is defined by the equation 
\[(\iota\otimes\pi_l)W=W_{12}\Ind(W_1)_{13},\] where the leg numbering notation on the right hand side of this equation is to be understood in the sense of the identification $\mathcal{F}\cong \W^2(\mbG)\otimes\Ind(\C_0(\widehat\mbG_1))$. 
The isomorphism $\Phi$ defined above intertwines the the strict $*$-homomorphism  $\pi_l$ with the $\Ad_{\Ind(W_1)}$ (see Proposition 3.7 and Section 4 of \cite{Vaes}):
\begin{equation}\label{phident}\Phi\pi_l(x)\Phi^*=\Ind(W_1)(x\otimes 1)\Ind(W_1)^*.\end{equation}
Let us write the corepresentation equation $(\Delta_{\mbG}\otimes\iota)\Ind(W_1)=\Ind(W_1)_{13}\Ind(W_1)_{23}$ in terms of $W$ and $\Ind(W_1)$:
\begin{equation}\label{indform}
\Ind(W_1)_{23}W_{12}\Ind(W_1)_{23}^*=W_{12}\Ind(W_1)_{13}
\end{equation}
Applying the character $\chi_g:\C_0(\G)\rightarrow\mathbb{C}$, $g\in \G$ to the first leg of the above equation we get the following formula
\begin{equation}\label{indform1}\Ind(W_1)(L_g\otimes 1)\Ind(W_1)^*=L_g\otimes \Ind(W_1)(g),\end{equation} where we identified the induced corepresentation $\Ind(W_1)\in\mathcal{L}(\C_0(\mbG)\otimes\Ind(\C_0(\widehat\mbG_1)))$ with the corresponding representation of the group $\G$ on $\mathcal{K}(\Ind(\C_0(\widehat\mbG_1)))$:
\[\Ind(W_1)(g)=(\chi_g\otimes\iota)(\Ind(W_1)).\]

The twisted objects are defined similarly. Starting with the strict $*$-homomorphism $\Theta^\Psi:\B(\W^2(\mbG))\rightarrow\mathcal{L}(\mathcal{F}^\Psi)$ we may define the induced $\C^*$-module $\Ind(\C_0(\widehat\mbG^\Psi_1))$ and the isomorphism $\Phi^\Psi:\mathcal{F}^\Psi\rightarrow \W^2(\mbG)\otimes\Ind(\C_0(\widehat\mbG^\Psi_1))$. Defining $\Ind(W^\Psi_1)$ by the equation
\[(\iota\otimes\pi^\Psi_l)W^\Psi=W^\Psi_{12}\Ind(W^\Psi_1)_{13}\] we may relate $\pi^\Psi_l$, $\Phi^\Psi$ and $\Ind(W^\Psi_1)$ by the formula:
\[\Phi^\Psi\pi^\Psi_l(x)\Phi^{\Psi*}=\Ad_{\Ind(W^\Psi_1)}(x)=\Ind(W^\Psi_1)(x\otimes 1)\Ind(W^\Psi_1)^*.\]

In order to  compare $\Theta$ and $\Theta^\Psi$ let us define $\tilde\Psi\in\mathcal{L}(\Ind(\C_0(\widehat\mbG_1))\otimes\W^2(\mbG))$ by:
\begin{equation}\label{indpsi}\Ind(W_1)_{12}\Psi_{13}\Ind(W_1)_{12}^*=\Psi_{13}\tilde\Psi_{23}.\end{equation}
The existence of such $\tilde\Psi$ follows from equation \eqref{indform1}.
Let us also introduce \begin{equation}\label{psi'}\tilde{\Psi}'=\Phi^*\Sigma\tilde{\Psi}\Sigma\Phi\in\mathcal{L}(\mathcal{F})\end{equation} where $\Sigma$ is the flip operation on the tensor product.
\begin{twr}\label{thmthepsi}
Let $\Theta:\B(\W^2(\mbG))\rightarrow\mathcal{L}(\mathcal{F})$ and $\Theta^\Psi:\B(\W^2(\mbG))\rightarrow\mathcal{L}(\mathcal{F}^\Psi)$ be the strict $*$-homomorphisms introduced above and let $U\in\mathcal{L}(\mathcal{F},\mathcal{F}^\Psi)$ be the unitary transformation introduced in Proposition \ref{prop5.1}. Let $\Psi'\in\mathcal{L}(\mathcal{F})$ be the element introduced in Eq. \eqref{psi'}. Then 
\begin{equation}\label{thpsi}\Theta^\Psi(x)=U\left(\Ad_{\tilde{\Psi}'}\circ\Theta\right)(x) U^*\end{equation} for any $x\in\B(\W^2(\mbG))$. In particular, $\Ind(\C_0(\widehat\mbG_1))$ and $\Ind(\C_0(\widehat\mbG^\Psi_1))$ are isomorphic, via the isomorphism  $K:\Ind(\C_0(\widehat\mbG_1))\rightarrow\Ind(\C_0(\widehat\mbG^\Psi_1))$ given by:
\begin{equation}\label{defk}\Ind(\C_0(\widehat\mbG_1))\ni v\mapsto K(v)=U\tilde{\Psi}'v\in\Ind(\C_0(\widehat\mbG^\Psi_1)).\end{equation}
\end{twr}
\begin{proof} Since the algebras $\W^\infty(\widehat{\mbG}^\Psi)'$ and $\W^\infty(\mbG^\Psi)'$ generate $\B(\W^2(\mbG))$, it is enough to check Eq. \eqref{thpsi} on them separately.
It is easy to see that for any $x\in\W^\infty(\widehat{\mbG}^\Psi)'$
\begin{equation}\label{thpsi1}(\Ad_{\tilde{\Psi}'}\circ\Theta)(\hat{J}x^*\hat{J})=U^*\pi_r(x)U,\end{equation} which proves  equality \eqref{thpsi} on $\W^\infty(\widehat{\mbG}^\Psi)'$.
To prove it on $\W^\infty(\mbG^\Psi)'$ it is enough to see that $(\Ad_{\tilde{\Psi}'}\circ\Theta\otimes\id)\hat{V}^\Psi=(U^*\otimes \id)Y^\Psi(U\otimes \id)$. In the following computation we shall identify  $\mathcal{F}$ with $\W^2(\mbG)\otimes\Ind(\C_0(\widehat\mbG_1))$ by means of the isomorphism $\Phi$ (see \eqref{phid}). Using \eqref{phid} and \eqref{indpsi} it is easy to show that under the aforementioned identification we have:
\begin{equation}\label{pilpsi}(\pi_l\otimes\id)\Psi=\Psi_{13}\tilde{\Psi}_{23}.\end{equation}
We compute
\begin{equation}\label{thpsi2}
\begin{array}{rl}
(\Ad_{\tilde{\Psi}'}\circ\Theta\otimes\id)(\hat{V}^\Psi_{13})
&=\tilde{\Psi}'_{12}\Psi^*_{13}\hat{V}_{13}\Psi_{13}'^*\tilde{\Psi}'^*_{12}\\&
=\tilde{\Psi}'_{12}\Psi_{13}^*\hat{V}_{13}\tilde{\Psi}'^*_{12}\Psi'^*_{13}\\&
=\tilde{\Psi}'_{12}\tilde{\Psi}'^*_{12}\tilde{\Psi}_{23}^*\Psi_{13}^*\hat{V}_{13}\Psi_{13}'^*\\&=(\pi_l\otimes\id)(\Psi^*)\hat{V}_{13}(\pi_r\otimes\id)(\Psi'^*)=Y^\Psi,
\end{array}
\end{equation}
where in the third equality we used an easy to check formula:
\[\hat{V}_{13}\tilde{\Psi}'^*_{12}\hat{V}_{13}^*=\tilde{\Psi}'^*_{12}\tilde{\Psi}_{23}^*\]
and in the fourth equality we used \eqref{pilpsi}.
\end{proof}
Let us now move on to the comparison of the induced corepresentations $\Ind(W_1^\Psi)$ and $\Ind(W_1)$. In order to do that we introduce two elements $\check{\Psi},\check{\Psi}'\in\mathcal{L}(\W^2(\mbG)\otimes\Ind(\C_0(\widehat\mbG_1)))$ such that 
\begin{align}\label{check1}(\iota\otimes\pi_l)(\Psi )&=\Psi_{12}\check{\Psi}_{13},\\
\label{check2}(\iota\otimes\pi_l)(\Psi )&=\Psi_{12}'\check{\Psi}'_{13}.
\end{align}
Their existence follows from the bicharacter equation for $\Psi$ and Eq. \eqref{indform1}. 
Let us note that $\check\Psi$ and the element $\tilde\Psi$ satisfying \eqref{pilpsi} are related by the formula
\begin{equation}\label{relatch}
\check\Psi=\Sigma\tilde\Psi^*\Sigma.
\end{equation}
We compute:
\begin{align*}(\iota\otimes\pi^\Psi_l)(W^\Psi)&=\check{\Psi}_{23}(\iota\otimes\pi_l)(\Psi W\Psi')\check{\Psi}_{23}^*\\
&=
\check{\Psi}_{23}\Psi_{12}\check{\Psi}_{13}W_{12}\Ind(W_1)_{13}\Psi_{12}'\check{\Psi}'_{13}\check{\Psi}^*_{23}\\&=
\Psi_{12}W_{12}\check{\Psi}_{13}\Psi_{12}'\Ind(W_1)_{13}\check{\Psi}'_{13}\\&=
\Psi_{12}W_{12}\Psi'_{12}\check{\Psi}_{13}\Ind(W_1)_{13}\check{\Psi}'_{13}\\&=
W^\Psi_{12}\check{\Psi}_{13}\Ind(W_1)_{13}\check{\Psi}'_{13},
\end{align*}
where in the third equality we used 
\begin{align*}
W_{12}\check{\Psi}_{13}W_{12}^*&=\check{\Psi}_{13}\check\Psi_{23},\\
\Ind(W_1)_{13}^*\Psi'_{12}\Ind(W_1)_{13}&=\Psi'_{12}\check\Psi_{23}^*.
\end{align*}
The first of these equalities follows from the bicharacter equation for $\Psi$ whereas the second one follows from the equality below:
\[\Ind(W_1)^*(R_g\otimes 1)\Ind(W_1)=R_g\otimes \Ind(W_1)(g).\]
This shows that with the identification $\Ind(\C_0(\widehat\mbG_1))\cong\Ind(\C_0(\widehat\mbG^\Psi_1))$ of Theorem \ref{thmthepsi} we have \[\Ind(W_1^\Psi)=\check{\Psi}\Ind(W_1)\check{\Psi}',\] where the product on the right hand side is taken in the $\C^*$-algebra $\mathcal{L}(\W^2(\mbG)\otimes\Ind(\C_0(\widehat\mbG_1))$.   

Finally, we shall  compare the  induced right coactions \begin{align*}\Ind(\beta)^\Psi:&\Ind(\C_0(\widehat\mbG^\Psi_1))\rightarrow\M(\Ind(\C_0(\widehat\mbG^\Psi_1))\otimes\C_0(\widehat\mbG^\Psi)),\\\Ind(\beta):&\Ind(\C_0(\widehat\mbG_1))\rightarrow\M(\Ind(\C_0(\widehat\mbG_1))\otimes\C_0(\widehat\mbG)).\end{align*}
Equation (6.2) of \cite{Vaes} defines the induced coaction $\Ind(\beta)$ as
\[(\iota\otimes\Ind(\beta))\Phi(i\tens_{\pi_{\widehat\mbG_1}}x)=W_{13}(\Phi\otimes\iota)\left((i\otimes 1)\tens_{\pi_{\widehat\mbG_1}\otimes\iota}W_{13}^*(\iota\otimes\beta)(x)\right).\] Similarly, $\Ind(\beta^\Psi)$ is defined by:
\[(\iota\otimes\Ind(\beta^\Psi))\Phi^\Psi(i\tens_{\pi_{\widehat\mbG^\Psi_1}}x)=W^\Psi_{13}(\Phi^\Psi\otimes\iota)\left((i\otimes 1)\tens_{\pi_{\widehat\mbG^\Psi_1}\otimes\iota}W^{\Psi*}_{13}(\iota\otimes\beta^\Psi)(x)\right).\]
Combining \eqref{phid}, \eqref{defk} and \eqref{relatch} one can easily check that \begin{equation}\label{kphi}(\iota\otimes K)\check\Psi\Phi=\Phi^{\Psi}U.\end{equation}
Using this formula we compute:
\begin{align*}(\iota\otimes\Ind(\beta^\Psi))(\iota\otimes K)\check\Psi\Phi(i\tens_{\pi_{\widehat\mbG_1}}x)&\\&\hspace{-4cm}=
(\iota\otimes\Ind(\beta^\Psi))\Phi^\Psi U(i\tens_{\pi_{\widehat\mbG_1}}x)&\\&\hspace{-4cm}=
W^\Psi_{13}(\Phi^\Psi\otimes\iota)\left((i\otimes 1)\tens_{\pi_{\widehat\mbG^\Psi_1}\otimes\iota}W^{\Psi*}_{13}(\iota\otimes\beta^\Psi)(\Psi x)\right)\\&\hspace{-4cm}=
W^\Psi_{13}(\Phi^\Psi\otimes\iota)\left((i\otimes 1)\tens_{\pi_{\widehat\mbG^\Psi_1}\otimes\iota}W^{\Psi*}_{13}\Psi_{13}\Psi_{12}(\iota\otimes\beta^\Psi)(x)\right)
\\&\hspace{-4cm}=
W^\Psi_{13}(\Phi^\Psi\otimes\iota)\left((i\otimes 1)\tens_{\pi_{\widehat\mbG^\Psi_1}\otimes\iota}\Psi'^*_{13}W^*_{13}\Psi_{12}(\iota\otimes\beta^\Psi)(x)\right)
\\&\hspace{-4cm}=
W^\Psi_{13}(\Phi^\Psi\otimes\iota)\left((i\otimes 1)\tens_{\pi_{\widehat\mbG^\Psi_1}\otimes\iota}\Psi'^*_{13}W^*_{13}\Psi_{12}\Psi_{23}^*(\iota\otimes\beta)(x)\Psi_{23}\right)\\&\hspace{-4cm}=
W^\Psi_{13}(\Phi^\Psi\otimes\iota)\left((i\otimes 1)\tens_{\pi_{\widehat\mbG^\Psi_1}\otimes\iota}\Psi'^*_{13}\Psi_{12}W^*_{13}(\iota\otimes\beta)(x)\Psi_{23}\right)\\
&\hspace{-4cm}=
W^\Psi_{13}(\Phi^\Psi\otimes\iota)(U\otimes\id)\left((i\otimes 1)\tens_{\pi_{\widehat\mbG_1}\otimes\iota}\Psi'^*_{13}W^*_{13}(\iota\otimes\beta)(x)\Psi_{23}\right)\\
&\hspace{-4cm}=
\Psi_{13}W_{13}(\Phi^\Psi\otimes\iota)(U\otimes\id)\left((i\otimes 1)\tens_{\pi_{\widehat\mbG_1}\otimes\iota}W^*_{13}(\iota\otimes\beta)(x)\Psi_{23}\right)\\
&\hspace{-4cm}=
\Psi_{13}W_{13}(\iota\otimes K\otimes\iota)(\check\Psi\otimes 1)(\Phi\otimes\id)\left((i\otimes 1)\tens_{\pi_{\widehat\mbG_1}\otimes\iota}W^*_{13}(\iota\otimes\beta)(x)\Psi_{23}\right)\\
&\hspace{-4cm}=
(\iota\otimes K\otimes\iota)\Psi_{13}\check\Psi_{12}\tilde\Psi_{23}^*W_{13}(\Phi\otimes\id)\left((i\otimes 1)\tens_{\pi_{\widehat\mbG_1}\otimes\iota}W^*_{13}(\iota\otimes\beta)(x)\right)\Psi_{23}
\\
&\hspace{-4cm}=
(\iota\otimes K\otimes\iota)\Psi_{13}\check\Psi_{12}\tilde\Psi_{23}^*(\iota\otimes\Ind(\beta))(\Phi(\iota\tens_{\pi_{\widehat\mbG_1}}x))\Psi_{23}.
\end{align*}
Using the equality
\[(\iota\otimes\Ind(\beta))(\check\Psi)=\Psi_{13}\check\Psi_{12}\] we get 
\begin{align*}(\iota\otimes\Ind(\beta^\Psi))(\iota\otimes K)\check\Psi\Phi\left(i\tens_{\pi_{\widehat\mbG_1}}x\right)&\\&\hspace{-3cm}=(\iota\otimes K\otimes\iota)\tilde\Psi_{23}^*\left((\iota\otimes\Ind(\beta))\check\Psi\Phi\left(\iota\tens_{\pi_{\widehat\mbG_1}}x\right)\right)\Psi_{23}.
\end{align*} 
Hence we see that
\[( K^*\otimes\iota)\Ind(\beta^\Psi)K(v)=\tilde\Psi^*\Ind(\beta)(v)\Psi\] for any $v\in\Ind(\C_0(\widehat\mbG^\Psi_1))$. 
In particular, for any $x\in\mathcal{K}(\Ind(\C_0(\widehat\mbG^\Psi_1)))$ we have 
\begin{equation}\label{twcomp}( K^*\otimes\iota)\Ind(\beta^\Psi)(KxK^*)( K\otimes\iota)=\tilde\Psi^*\Ind(\beta)(x)\tilde\Psi.\end{equation}
In what follows we summarize the above considerations: 
\begin{twr}\label{indreg} Let $\C_0(\widehat\mbG^\Psi)$ be the $\C^*$-algebra of the quantum group $\widehat\mbG^\Psi$ treated as the $\C^*$-Hilbert module over itself. Let $W_1^\Psi\in\M(\C_0(\mbG^\Psi)\otimes\C_0(\widehat\mbG^\Psi))$ be the left regular corepresentation of $\mbG^\Psi$ and $\beta^\Psi$ be the right coaction of $\widehat\mbG^\Psi$ defined in \eqref{bpsi}. Let  $K:\Ind(\C_0(\widehat\mbG^\Psi_1))\rightarrow\Ind(\C_0(\widehat\mbG_1))$ be the isomorphism introduced in Eq. \eqref{defk} (in what follows we shall identify $\Ind(\C_0(\widehat\mbG^\Psi_1))$ with $\Ind(\C_0(\widehat\mbG_1))$ by means of $K$).

The induced coaction $\Ind(W_1^\Psi)\in\mathcal{L}(\W^2(\mbG)\otimes\Ind(\C_0(\widehat\mbG_1)))$ is given by 
\begin{equation}\label{thforw}\Ind(W_1^\Psi)=\check{\Psi}\Ind(W_1)\check{\Psi}'\end{equation} where $\check{\Psi},\check{\Psi}'\in\mathcal{L}(\W^2(\mbG)\otimes\Ind(\C_0(\widehat\mbG_1)))$ are the unitary elements defined by \eqref{check1} and \eqref{check2}.
The twisted induced coaction \[\Ind(\beta)^\Psi:\Ind(\C_0(\widehat\mbG^\Psi_1))\rightarrow\M(\Ind(\C_0(\widehat\mbG_1))\otimes\C_0(\widehat\mbG^\Psi))\] is related with its untwisted counterpart $\Ind(\beta)$ by the following formula
\begin{equation}\label{thforb}\Ind(\beta^\Psi)(v)=\tilde\Psi^*\Ind(\beta)(v)\Psi\end{equation}
for any $v\in\Ind(\C_0(\widehat\mbG_1))$, where $\tilde\Psi\in\mathcal{L}(\Ind(\C_0(\widehat\mbG_1))\otimes\C_0(\widehat\mbG^\Psi))$ is the unitary element defined by \eqref{pilpsi}.
\end{twr}
\end{section}
\begin{section}{Rieffel deformation of homogeneous spaces - the quotient case}\label{rdhsec}
Let $\G$ be a locally compact group, $\G_1$ its closed subgroup and let $\X=\G/\G_1 $ be the homogeneous space of left $\G_1$-cosets. The standard action of $\G$ on $\X$  gives rise to the $\mbG$-$\C^*$-algebra $\mbX=(\C_0(\X),\Delta_{\X})$. Suppose that $\Gamma\subset \G_1$ is a closed abelian subgroup and let $\Psi$ be a 2-cocycle on the dual group $\Hat\Gamma$. Using the results of paper \cite{Kasp2} described in Section \ref{rco} we know that $\mbX$ may be deformed to $\mbG^\Psi$-$\C^*$-algebra $\mbX^\Psi$.

On the other hand, applying the Rieffel deformation to the subgroup $\mbG_1\subset\mbG$ we get a closed quantum subgroup $\mbG_1^\Psi\subset\mbG^\Psi$ in the sense of Definition 2.5 of \cite{Vaes}. Let  $\mbG^\Psi/\mbG_1^\Psi$ be the $\C^*$-algebraic quotient space of Definition \ref{cqs}.
The aim of this section is to show that $\mbG^\Psi/\mbG_1^\Psi\cong\mbX^\Psi$.

Let $\iota\in\Mor(\C_0(\mbX),\C_0(\mbG))$ be the standard embedding - it maps a function $f\in\C_0(\mbX)$ to the same function on $\G$ which is constant on the right $\G_1$-cosets. Let $\alpha:\Gamma\rightarrow\Aut(\C_0(\mbX))$ and $\rho:\Gamma^2\rightarrow\Aut(\C_0(\mbG))$ be the actions introduced in Section \ref{rco}. The  action $\rho$ of $\Gamma^2$ on $\C_0(\mbG)$ restricts to the action $\alpha$ on $\C_0(\mbX)$. To be more precise, the second copy of $\Gamma$ in $\Gamma^2$ acts trivially on $\C_0(\mbX)$ and the action of the first copy coincides with $\alpha$.  By the results of  Section 3.2 of \cite{Kasp} the embedding $\iota$ may be twisted to the embedding $\iota^\Psi\in\Mor(\C_0(\mbX^\Psi),\C_0(\mbG^\Psi))$. The comultiplication   $\Delta_{\mbG^\Psi}$   restricts  to the coaction $\Delta_{\mbX^\Psi}$ on $\iota^\Psi(\C_0(\mbX^\Psi))$. The last statement follows from  Theorem 4.11 of \cite{Kasp} and the description of $\Delta_{\mbX^\Psi}$ given in Section \ref{rco}. In particular, $\Delta_{\mbX^\Psi}$ is implemented by the multiplicative unitary:
\begin{equation}\label{imwp} \Delta_{\mbX^\Psi}(a)=W^{\Psi*}(1\otimes a)W^{\Psi},\end{equation} for any $a\in\C_0(\mbX^\Psi)$. 

As was explained in Remark \ref{expr}, the crossed product of $\C_0(\mbG^\Psi/\mbG_1^\Psi)$ by the coaction of $\mbG^\Psi$ coincides with $\mathcal{K}(\Ind(\C_0(\widehat\mbG^\Psi_1)))$.  Using Theorem \ref{indreg} we see that we may identify $\mathcal{K}(\Ind(\C_0(\widehat\mbG^\Psi_1)))$ with $\mathcal{K}(\Ind(\C_0(\widehat\mbG_1)))$, which in turn may be identified with the $\C^*$-algebra $\mbG\ltimes\C_0(\mbG/\mbG_1)$. The former  $\C^*$-algebra may be identified with $B=[\C_0(\widehat\mbG)\C_0(\mbX)]\subset\B(\W^2(\mbG))$ - the $\C^*$-algebra generated  by $\C_0(\widehat\mbG)$ and $\C_0(\mbX)$ inside $\B(\W^2(\mbG))$.  

Under the identification $\mathcal{K}(\Ind(\C_0(\widehat\mbG^\Psi_1)))=B$ the crossed product structure of $\mathcal{K}(\Ind(\C_0(\widehat\mbG^\Psi_1)))$ may be described as follows. Using Eq. \eqref{thforb} one can see that the dual  coaction $\Ind(\beta^\Psi):B\rightarrow\M(B\otimes\C_0(\widehat\mbG^\Psi))$ is implemented by the unitary $\widehat{V}^\Psi$ introduced in \eqref{hatV}. Eq. \eqref{thforw} shows in turn that the induced corepresentation $\Ind(W_1^\Psi)\in\M(\C_0(\mbG^\Psi)\otimes B)$ can be identified with the regular corepresentation  $W^\Psi\in\M(\C_0(\mbG^\Psi)\otimes\C_0(\widehat\mbG^\Psi))\subset\M(\C_0(\mbG^\Psi)\otimes B)$. We are now well prepared to prove
\begin{twr} Let $\mbX^\Psi$ be the Rieffel deformation of the homogeneous space $\mbX$ and $\mbG^\Psi/\mbG^\Psi_1$ the Vaes' quotient considered above. Then $\mbX^\Psi\cong\mbG^\Psi/\mbG^\Psi_1$. 
\end{twr}
\begin{proof} By the universal properties of the crossed product $\C^*$-algebra $\Gamma\ltimes_\rho\C_0(\mbX)$ there exists a unique morphism  $\iota^\Gamma\in\Mor(\Gamma\ltimes_\rho\C_0(\mbX),B)$ which is identity on $\C_0(\mbX)\subset\M(\Gamma\ltimes_\rho\C_0(\mbX))$ and such that it sends the unitary generator $u_\gamma\in\M(\C^*(\Gamma))$ to the left shift $L_\gamma\in\M(B)$. Using Theorem 3.6, \cite{Kasp} we may conclude that the restriction of $\iota^\Gamma$ to $\C_0(\mbX^\Psi)$ gives rise to the injective morphism $\iota|_{\C_0(\mbX^\Psi)}$  which we shall denote by $\iota^\Psi\in\Mor(\C_0(\mbX^\Psi),B)$. 
Let $\hat\rho^\Psi$ be the twisted dual action on $\Gamma\ltimes_\rho\C_0(\mbX)$. In what follows we shall interpret it as a coaction $\hat\rho^\Psi\in\Mor(\Gamma\ltimes_\rho\C_0(\mbX),\Gamma\ltimes_\rho\C_0(\mbX)\otimes\C^*(\Gamma))$. Under this interpretation, the relation of $\hat\rho^\Psi$ with its untwisted counterpart $\hat\rho$ is established by the following formula:
\begin{equation}\label{twrho} \hat\rho^\Psi(x)=\Psi^*\rho(x)\Psi
\end{equation} for any $x\in\Gamma\ltimes_\rho\C_0(\mbX)$. Let us note that $\iota^\Gamma$ intertwines the twisted dual coactions:
\begin{equation}\label{intct}\Ind(\beta^\Psi)\circ\iota^\Gamma=(\iota^\Gamma\otimes\id)\circ\hat\rho^\Psi.\end{equation}
In order to see it  we have to observe that:
\begin{itemize}
\item $\iota^\Gamma$ intertwines $\hat\rho$ and $\Ind(\beta)$: \begin{equation}\label{intnt}\Ind(\beta)\circ\iota^\Gamma=(\iota^\Gamma\otimes\id)\circ\hat\rho.\end{equation} Indeed, the equality $\Ind(\beta)(\iota^\Gamma(f))=(\iota^\Gamma\otimes\id)(\hat\rho(f))$ for any $f\in\C_0(\mbX)$ is the consequence of the the simultaneous invariance of $f$ under $\hat\rho$ and $\Ind(\beta)$. Moreover $\Ind(\beta)\circ\iota^\Gamma(u_\gamma)=L_\gamma\otimes L_\gamma=(\iota^\Gamma\otimes\id)\circ\hat\rho(u_\gamma)$ for any $\gamma\in\Gamma$. Using the fact $\C_0(\mbX)$ and $\C^*(\Gamma)$ generate $\Gamma\ltimes_\rho\C_0(\mbX)$ we get a proof of \eqref{intnt}.
\item Let $\tilde\Psi$ be the unitary element introduced in \eqref{indpsi}. Note that $(\iota^\Gamma\otimes\id)\Psi=\tilde\Psi$. 
\end{itemize}
Using the above observations and  Eqs. \eqref{thforb}, \eqref{twrho} one easily gets \eqref{intct} which in turn implies that $\iota^\Psi(\C_0(\mbX^\Psi))\subset\M(B)^{\Ind(\beta^\Psi)}$. 
The proof of the equality  $\iota^\Psi(\C_0(\mbX^\Psi))=\C_0(\mbG^\Psi/\mbG_1^\Psi)$ is based on Theorem 6.7 of \cite{Vaes}. Its application requires the check of the following two conditions:
\begin{itemize}
\item The map $x\mapsto W^{\Psi*}(1\otimes x)W^\Psi$ defines a continuous coaction on $\iota^\Psi(\C_0(\mbX^\Psi))$. Indeed, this follows from \eqref{imwp}.
\item We have $[\iota^\Psi(\C_0(\mbX^\Psi))\C_0(\widehat\mbG)]=B$. In order to see that we compute \begin{align*}[\iota^\Psi(\C_0(\mbX^\Psi))\C_0(\widehat\mbG)]&=[\iota^\Psi(\C_0(\mbX^\Psi))\C^*(\Gamma)\C_0(\widehat\mbG)]=[\iota^\Gamma(\Gamma\ltimes_\rho\C_0(\mbX^\Psi))\C_0(\widehat\mbG)]\\&=[\iota^\Gamma(\Gamma\ltimes_\rho\C_0(\mbX))\C_0(\widehat\mbG)]=[\iota^\Gamma(\C_0(\mbX))\C^*(\Gamma)\C_0(\widehat\mbG)]\\&=[\C_0(\mbX)\C_0(\widehat\mbG)]=B\end{align*} where in the first and the fourth equality we used the fact that $[\C^*(\Gamma)\C_0(\widehat\mbG)]=\C_0(\widehat\mbG)$ and in the third equality we used the fact that $\Gamma\ltimes_\rho\C_0(\mbX^\Psi)=\Gamma\ltimes_\rho\C_0(\mbX)$ (note that the action of $\Gamma$ on $\C_0(\mbX)$ and on $\C_0(\mbX^\Psi)$ is denoted by the same $\rho$). 
\end{itemize} This ends the proof of the isomorphism $\mbX^\Psi\cong \mbG^\Psi/\mbG^\Psi_1$.  
\end{proof}
\end{section}
\begin{section}{Rieffel deformation of $\mbG$-simple $\C^*$-algebras}\label{nonq}
Let $\mbG$ be a locally compact quantum group corresponding to a locally compact group $\G$.
The aim of this section is to prove that the Rieffel deformation $\mbX^\Psi$ of a $\mbG$-simple $\C^*$-algebra $\mbX$ is $\mbG^\Psi$-simple (see Definition \ref{homspdef}). In particular, the Rieffel deformation of a quotient space $\X=\G/\G_1$ (also in the case when $\Gamma$ is a subgroup of $\G$ but not of $\G_1$) is $\mbG^\Psi$-simple. The idea of the following proof  is based in the functorial properties of the Rieffel deformation (see Section 3.2, \cite{Kasp}).
\begin{twr} Let $\mbX$ be a $\mbG$-$\C^*$-algebra. Let $\Gamma$ be a closed abelian subgroup of $\mbG$ and $\Psi$ a 2-cocycle on the dual group $\hat\Gamma$. The Rieffel deformation $\mbX^\Psi$ of $\mbX$ is  $\mbG^\Psi$-simple.
\end{twr}
\begin{proof}
Let $\mathbb{B}$ be a $\mbG^\Psi$-$\C^*$-algebra and let $\pi\in\Mor_{\mbG^\Psi}(\mbX^\Psi,\mathbb{B})$. The Rieffel deformation $\Gamma^\Psi$ of $\Gamma$ is a quantum closed subgroup of $\mbG^\Psi$. Actually, $\Gamma^\Psi=\Gamma$ which easily follows from the abelianity of $\Gamma$ (see Appendix B of \cite{Kasp3}). Let $\alpha:\Gamma\rightarrow\Aut(\C_0(\mbX^\Psi))$ be the action that corresponds to the $\Gamma$-restriction of the coaction $\Delta_{\mbX^\Psi}$. Similarly, we may introduce $\beta:\Gamma\rightarrow\Aut(\C_0(\mathbb{B}))$. Obviously the morphism $\pi$ intertwines $\alpha$ and $\beta$:
$\pi\circ\alpha_{\gamma}=\beta_{\gamma}\circ\pi$ for any $\gamma\in\Gamma$. The deformation data $(\C_0(\mathbb{B}),\beta,\bar\Psi)$ gives rise to the deformed  $\C^*$-algebra that we shall denote by $\C_0(\mathbb{B}^{\bar\Psi})$ and the deformed morphism $\pi^{\bar\Psi}\in\Mor(\C_0(\mbX),\C_0(\mathbb{B}^{\bar\Psi}))$. Note that $(\mbG^{\Psi})^{\bar\Psi}=\mbG$. Using the ideas presented in Section \ref{rco} we may construct the twisted coaction $\Delta_{\mathbb{B}}^{\bar\Psi}$ of $\mbG$ on $\mathbb{B}^{\bar\Psi}$ and check that $\pi^{\bar\Psi}$ is a $\mathbb{G}$-morphism. 
The $\mbG$-simplicity of $\mbX$ implies that $\ker\pi^{\bar\Psi}=\{0\}$. Using Proposition 3.8 \cite{Kasp} we see that $\ker\pi=\{0\}$, which ends the proof of $\mbG^\Psi$-simplicity of $\mbX^\Psi$ .
\end{proof}
\end{section}

\end{document}